\UseRawInputEncoding
\documentclass[leqno]{amsart}

\usepackage{stmaryrd,graphicx}
\usepackage{amssymb,mathrsfs,amsmath,amscd,amsthm,color}
\usepackage{float}
\usepackage{mathabx}
\usepackage[all,cmtip]{xy}
\DeclareMathAlphabet{\mathpzc}{OT1}{pzc}{m}{it}
\usepackage{amsfonts,latexsym,wasysym}

\DeclareMathOperator{\im}{Im}
\newcommand{\lpc}{\mathbf{lpc}}
\newcommand{\san}{\mathscr{S}}
\newcommand{\wildn}{\mathbf{w}_{n}}
\newcommand{\wild}{\mathbf{w}}
\newcommand{\rg}{\mathbf{rg}}
\newcommand{\rgn}{\mathbf{rg}_{n}}
\newcommand{\sw}{\bigcurlyvee}

\newcommand{\ui}{[0,1]}
\renewcommand{\int}{\text{int}}

\newcommand{\wt}{\widetilde}

\newcommand{\mcc}{\mathcal{C}}

\newcommand{\scrb}{\mathscr{B}}

\newcommand{\scru}{\mathscr{U}}

\newcommand{\scrs}{\mathscr{S}}

\newcommand{\bbd}{\mathbb{D}}

\newcommand{\bbe}{\mathbb{E}}
\newcommand{\bbn}{\mathbb{N}}

\newcommand{\bbr}{\mathbb{R}}

\newcommand{\bbz}{\mathbb{Z}}

\newcommand{\ov}{\overline}

\newtheorem{theorem}{Theorem}[section]
\newtheorem{lemma}[theorem]{Lemma}
\newtheorem{proposition}[theorem]{Proposition}
\newtheorem{corollary}[theorem]{Corollary}
\theoremstyle{definition}\newtheorem{definition}[theorem]{Definition}
\newtheorem{example}[theorem]{Example}

\newtheorem{remark}[theorem]{Remark}
\newtheorem{problem}[theorem]{Problem}

\begin{document}
\title[Higher-homotopy wild sets]{Higher-homotopy wild sets}

\author[J. Brazas]{Jeremy Brazas}
\address{West Chester University\\ Department of Mathematics\\
West Chester, PA 19383, USA}
\email{jbrazas@wcupa.edu}

\author[A. Mitra]{Atish Mitra}
\address{Montana Technical University\\ Department of Mathematical Sciences\\
1300 West Park Street Butte, MT 59701, USA}
\email{amitra@mtech.edu}

\date{\today}

\begin{abstract}
The $\pi_n$-wild set $\mathbf{w}_{n}(X)$ of a topological space $X$ is the subspace of $X$ consisting of the points at which there exists a shrinking sequence of essential based maps $S^n\to X$. In this paper, we show that the homotopy type of $\mathbf{w}_{n}(X)$ is a homotopy invariant of $X$ and, in analogy to the known one-dimensional case, we show that for certain $n$-dimensional $\pi_n$-shape injective metric spaces, the homeomorphism type of $\mathbf{w}_{n}(X)$ is a homotopy invariant of $X$. We also prove that the $\pi_n$-wild set of a Peano continuum can be homeomorphic to any compact metric space.
\end{abstract}

\maketitle

\section{Introduction}

There are many potential definitions of ``wild point" in a topological space. From the perspective of algebraic topology, if a space has wild points, e.g. if it fails to be locally contractible or semilocally simply connected, then some standard methods in homotopy theory fail to apply to the space in question. Notably, the Whitehead Theorem stating that ``weakly homotopy equivalence spaces are homotopy equivalent" \cite{whiteheadtheorem} may not be applicable. On the other hand, the existence of wild points is surprisingly helpful for distinguishing and classifying homotopy types of some families of Peano continua (compact, connected, locally connected metric spaces). It is a major achievement in the algebraic topology of locally complicated spaces that the Whitehead Theorem holds for one-dimensional Peano continua \cite{Edaonedim} and planar Peano continua \cite{Kent}. The following question remains open.

\begin{problem}\label{whproblem}
Does the Whitehead Theorem hold for all finite dimensional Peano continua? That is, if $f:X\to Y$ is a weak homotopy equivalence of finite-dimensional Peano continua, must $f$ be a homotopy equivalence?
\end{problem}

Problem \ref{whproblem} has a negative answer for general Peano continua as infinite dimensional, non-contractible spaces with trivial homotopy groups are constructed by Karimov and Repov\v s in \cite{krhawaiiangroups} and \cite{krnoncontractible}. In \cite[Problem 5.1]{krnoncontractible}, Karimov and Repov\v s ask if there exists a \textit{finite dimensional} non-contractible Peano continuum with trivial homotopy groups. The existence of such a space would answer Problem \ref{whproblem} in the negative but no counterexample has been produced so far. 

A key to proving the Whitehead Theorem in the one-dimensional and planar cases is the use of the ``wild" or ``bad" set $\wild_1(X)$ consisting of all points in $X$ at which $X$ fails to be semilocally simply connected (and denoted various ways in the literature). For one-dimensional and planar $X$, the wild set $\wild_1(X)$ is rigid in the sense that it is fixed under all maps $X\to X$ that are homotopic to the identity map. Moreover, this rigidity implies that the homeomorphism type of $\wild_1(X)$ is a homotopy invariant of $X$. In fact, for one-dimensional spaces $X$ where all points are wild points, the isomorphism type of the fundamental group $\pi_1(X)$ alone determines the entire homeomorphism type of the space \cite{ConnerEda,ConnerEda2}. These invariance results are implicit to the arguments used in \cite{ConnerKent,EdaSpatial} and are explicitly stated and proved in \cite[Section 9]{BFPantsspace}. In addition to the depth of applications in low-dimensional settings, wild sets also have utility in broader contexts since, in general, the \textit{homotopy type} of $\wild_1(X)$ is a homotopy invariant of $X$ \cite[Theorem 2.11]{Braztransprodreduc}. Whether one seeks to answer Problem \ref{whproblem} in the affirmative or negative, the successful one-dimensional and planar theories suggest the relevancy of higher-dimensional wild sets.

In this paper, we define and study subspaces of a given space $X$ that consist of points where algebraic wildness occurs in higher-dimensional homotopy groups. We say that a point $x\in X$ is a \textit{$\pi_n$-wild point} if there exists a sequence of essential, i.e. non-null-homotopic, maps $f_k:S^n\to X$, $k\in\bbn$ based at $x$ that converge to the constant map at $x$ in the compact-open topology. To simplify this concept, we note that such sequences can be adjoined to form what we call a ``fully essential" based map $f:(\bbe_n,b_0)\to (X,x)$ on the $n$-dimensional infinite earring space $\bbe_n$. The \textit{$\pi_n$-wild set} of $X$ is the subspace $\wildn(X)$ of $X$ consisting of all $\pi_n$-wild points of $X$.

In Sections \ref{section2wildset} and \ref{section3properties} we establish various properties and examples relevant to $\pi_n$-wild sets. In Section \ref{section4homotopyinvariance}, we prove that the homotopy type of $\wildn(X)$ is a homotopy invariant of $X$, that is, if $X\simeq Y$, then $\wildn(X)\simeq \wildn(Y)$ (see Theorem \ref{homotopyinvariance}) and we use this fact to distinguish homotopy types without directly appealing to uncountable algebraic invariants. In Section \ref{section5makingspaceswild}, we use ``shrinking point-attachment spaces," similar to those applied in \cite{Edamakingspaceswild,EdaHigasikawa}, to prove the following theorem. In particular, this result shows that the $\pi_n$-wild set of a Peano continuum may be an arbitrary compact metric space.

\begin{theorem}\label{thm1}
Let $n\geq 1$. If $X$ is a Peano continuum, then $\wildn(X)$ is a compact metric space. Moreover, if $C$ is any compact metric space, then there exists a Peano continuum $X$ such that
\begin{enumerate}
\item $\wildn(X)=C$,
\item $X\backslash C$ is a countable disjoint union of open $1$-cells and open $n$-cells,
\item $\dim(X)=\max\{\dim(C),n\}$.
\end{enumerate}
\end{theorem}

In Section \ref{section6rigidity}, we extend the established one-dimensional theory by showing that higher $\pi_n$-wild sets are ``rigid" for certain $n$-dimensional spaces. We say that a space $X$ is \textit{$\pi_n$-rigid at} $x\in X$ if there exists a fully essential map $f:(\bbe_n,b_0)\to (X,x)$ that cannot be freely homotoped away from the point $x$, i.e. if every homotopy $F:\bbe_n\times \ui\to X$ extending $f$ by $F(a,0)=f(a)$ must satisfy $F(b_0,1)=x$. We set $\mathbf{rg}_n(X)=\{x\in X\mid X\text{ is }\pi_n\text{-rigid at }x\}$ and say that $X$ is \textit{completely $\pi_n$-rigid} if $X$ is $\pi_n$-rigid at every $\pi_n$-wild point, i.e. if $\rgn(X)=\wildn(X)$. The main utility of this last concept is that if $X$ and $Y$ are homotopy equivalent and both $X$ and $Y$ are completely $\pi_n$-rigid spaces, then $\wildn(X)$ and $\wildn(Y)$ are homeomorphic (Theorem \ref{rigidthm1}). As noted above, it is known that every one-dimensional metric space is completely $\pi_1$-rigid.

In the one-dimensional and planar settings, one can readily distinguish fundamental group elements using the fact that such spaces are $\pi_1$-shape injective, that is, the canonical homomorphism $\pi_1(X)\to \check{\pi}_1(X)$ from the fundamental group to the first shape homotopy group is always injective \cite{CConedim,FZ05}. However, for $n\geq 2$, an $n$-dimensional Peano continuum $X$ need not be $\pi_n$-shape injective (the canonical homomorphism $\pi_n(X)\to \check{\pi}_n(X)$ need not be injective) \cite{EKR,Felt,KRsuspension}. Hence, to identify a higher-dimensional analogue of the result that one-dimensional spaces are completely $\pi_1$-rigid, we restrict to $n$-dimensional, $\pi_n$-shape injective Peano continua. We prove the following in Section \ref{section6rigidity}.

\begin{theorem}\label{thm2}
Let $n\geq 1$. If $X$ is an $n$-dimensional, $\pi_n$-shape injective Peano continuum that can be expressed as an inverse limit of a sequence of compact $n$-dimensional polyhedra $K_n$ with $(n-1)$-connected universal covers, then $X$ is completely $\pi_n$-rigid.
\end{theorem}

The authors do not know if the hypothesis that $X$ is $\pi_n$-shape injective can be weakened when $n\geq 2$. However, we note in Example \ref{connectednessexample} why the higher connectedness hypothesis on the polyhedra $K_n$ cannot be removed. Finally, since $2$-complexes always have $1$-connected universal covers, we have the following special case of interest.

\begin{corollary}\label{cor2}
If $X$ is a $2$-dimensional, $\pi_2$-shape injective Peano continuum, then $X$ is completely $\pi_2$-rigid. In particular, if $X$ and $Y$ are homotopy equivalent $2$-dimensional, $\pi_2$-shape injective Peano continua, then $\wild_2(X)\cong \wild_2(Y)$.
\end{corollary}

\section{The $\pi_n$-wild set of a space}\label{section2wildset}

Unless otherwise stated, all topological spaces are assumed to be Hausdorff and a ``map" is a continuous function. Throughout, $S^n$ will be the unit $n$-sphere with basepoint $s_0=(1,0,\dots,0)$. A map $f:S^n\to X$ is said to be \textit{inessential} if it is null-homotopic and \textit{essential} otherwise.


When $X$ and $Y$ are spaces, $Y^X$ will denote the space of continuous functions $X\to Y$ with the compact-open topology. If $A\subseteq X$ and $B\subseteq Y$, then $(Y,B)^{(X,A)}$ denotes the subspace of $Y^X$ consisting of relative maps $(X,A)\to (Y,B)$. When $y\in Y$, $c_y:X\to Y$ will denote the constant map at $y$. For a based topological space $(X,x_0)$, we write $\Omega^{n}(X,x_0)$ to denote the $n$-loop space $(X,x_0)^{(S^n,s_0)}$ and $\pi_n(X,x_0)=\{[f]\mid f\in \Omega^n(X,x_0)\}$ to denote the $n$-th homotopy group. When the basepoint is clear from context, we may simplify this notation to $\Omega^n(X)$ and $\pi_n(X)$. We say that a sequence $\{f_k\}_{k\in \bbn}$ of maps $f_k:X\to Y$ \textit{converges to} $y\in Y$ if $\{f_k\}_{k\in\bbn}\to c_y$ in $Y^X$, that is, if for every neighborhood $U$ of $y$, there exists $K\in\bbn$ such that $\im(f_k)\subseteq U$ for all $k\geq K$.

\begin{definition}
The \textit{shrinking wedge} of countable set $\{(A_j,a_j)\}_{j\in J}$ of based spaces is the space $\sw_{j\in J}(A_j,a_j)$ whose underlying set is the usual one-point union $\bigvee_{j\in J}(A_j,a_j)$ with canonical basepoint $b_0$ and where $A_j$ is identified canonically as a subset. A set $U$ is open in $\sw_{j\in J}A_j$ if
\begin{itemize}
\item $U\cap A_j$ is open in $A_j$ for all $j\in J$,
\item and whenever $b_0\in U$, we have $A_j\subseteq U$ for all but finitely many $j\in J$.
\end{itemize}
When the basepoints and/or indexing set are clear from context, we may write the shrinking wedge as $\sw_{J}A_j$. 
\end{definition}

The \textit{$n$-dimensional infinite earring space} is the shrinking wedge $\bbe_n=\sw_{j\in\bbn}S^n$ of $n$-spheres. We identify $\bbe_0=\sw_{\bbn}(S^0,1)$ with the space $\{1,1/2,1/3,\dots,0\}$ consisting of a single convergent sequence and basepoint $0$. Let $\ell_j:S^n\to \bbe_n$ denote the inclusion of the $j$-th sphere. When $n\geq 2$, it is known that $\bbe_n$ is $(n-1)$-connected, locally $(n-1)$-connected and that the canonical map $\Psi_n:\pi_n(\bbe_n)\to \check{\pi}_{n}(X)\cong\prod_{j\in\bbn}\bbz$ to the $n$-th shape homotopy group is an isomorphism \cite{EK00higher}.

\begin{definition}
For a map $f\in (X,x)^{(\bbe_n,b_0)}$, we will refer to $f_j=f\circ \ell_j$ as the \textit{$j$-th restriction} of $f$. We say that a map $f:\bbe_n\to X$ is \textit{fully essential} if the $j$-th restriction $f_j:S^n\to X$ is essential for all $j\in\bbn$.
\end{definition}

\begin{remark}\label{mappingspaceremark}
Exponential laws for spaces imply that for any based space $(Y,y)$, there is a canonical bijection $(Y,y)^{(\bbe_n,b_0)}\cong (\Omega^n(Y,y),c_{y})^{(\bbe_0,0)}$ given by $f\mapsto \{f_j\}_{j\in\bbn}$. In other words, maps $\bbe_n\to Y$ based at $y$ are in bijective correspondence with sequences of based maps $S^n\to Y$ that converge to $y$.
\end{remark}

\begin{definition}\label{wildsetdef}
A point $x\in X$ is a \textit{$\pi_n$-wild point} if there exists a fully essential map $f:(\bbe_n,b_0)\to (X,x)$. The \textit{$\pi_n$-wild set} of $X$ is the subspace $\wildn(X)$ of $X$ consisting of all $\pi_n$-wild points of $X$.
\end{definition}

\begin{remark}
There are other variations of wild sets that may be preferable depending on the context. 
\begin{enumerate}
\item A point $x\in X$ is a sequential-based $\pi_n$-wild point if there exists a sequence of essential based maps $\alpha_n:(S^n,s_0)\to (X,x)$ that converge to $x$,
\item A point $x\in X$ is a sequential-free $\pi_n$-wild point if there exists a sequence of essential maps $\alpha_n:S^n\to X$ that converge to $x$ (but which are not necessarily based at $x$),
\item A point $x\in X$ is a topological-based $\pi_n$-wild point if for every neighborhood $U$ of $x$, the homomorphism $\pi_n(U,x)\to \pi_n(X,x)$ induced by the inclusion map is non-trivial,
\item A point $x\in X$ is a topological-free $\pi_n$-wild point if for every neighborhood $U$, there exists a map $\alpha:S^n\to U$ that is essential in $X$.
\end{enumerate}
Variation (1) is equivalent to Definition \ref{wildsetdef} and is our preferred definition. In general, all four variations of $\pi_n$-wild sets are distinct. When $X$ is first countable, we have equivalences (1) $\Leftrightarrow$ (3) and (2) $\Leftrightarrow$ (4). When $X$ is locally path-connected, we have equivalence (3) $\Leftrightarrow$ (4). Other notions of wildness defined in terms of (co)homology groups may also be defined. We choose to focus on Variation (1) since it is most directly related to infinite-product algebra in the $n$-th homotopy group. For instance, if $\omega_1+1=\omega_1\cup\{\omega_1\}$ is the first compact uncountable ordinal with basepoint $\omega_1$, then the basepoint of the $n$-th reduced suspension $\Sigma^{n}(\omega_1+1)$, $n\geq 2$ satisfies (1) but not (3). This is reflected in the fact that $\pi_n(\Sigma^{n}(\omega_1+1))$ is completely tame. In fact, one can show it is free abelian and admits no non-trivial infinite sums.
\end{remark}

\begin{example}\label{emptyexample}
If $X$ is locally contractible at $x\in X$, then $x\notin\wildn(X)$. Hence, if $X$ is a locally contractible space, e.g. if $X$ is a CW-complex or manifold, then $\wildn(X)=\emptyset$.
\end{example}

\begin{example}\label{earringexample}
If $m\geq n$ and $0\neq [g]\in \pi_m(S^n)$, then we can define a fully essential map $f:\bbe_m\to \bbe_n$, which maps the $j$-th sphere of $\bbe_m$ to the $j$-th sphere of $\bbe_n$ by the map $g$. Hence, $\wild_{m}(\bbe_n)=\{b_0\}$ whenever $\pi_m(S^n)\neq 0$. For instance, this occurs when $n\in\{2,3,4,5\}$ since is known that $\pi_m(S^n)\neq 0$ for all $m\geq n$ \cite{CurtisEhomotopy,pi2,Mahowald,Mori}.
\end{example}

\begin{remark}[Cardinality]
The existence of a $\pi_n$-wild point in a path-connected compact metric space $X$ directly effects the cardinality of $\pi_n(X,x_0)$. It is proved in \cite{Pawlikowski} that if $X$ is a path-connected compact metric space and there exists a fully essential map $f:\bbe_1\to X$, then the image of the induced homomorphism $f_{\#}:\pi_1(\bbe_1,b_0)\to\pi_1(X,f(b_0))$ is uncountable (note that Pawlikowski's proof provides an alternative to Shelah's forcing proof in \cite{Shelah}). Pawlikowski's argument is modified to apply to higher homotopy groups in \cite{corsonthesis}. Hence, if $X$ is a path-connected compact metric space $X$ and $\wildn(X)\neq \emptyset$, then $\pi_n(X,x_0)$ is uncountable. We point out in Example \ref{uncountableexample} that it is possible for the n-th homotopy group of an $n$-dimensional Peano continuum to be uncountable even if it has no $\pi_n$-wild points. However, Corson also shows in \cite{corsonthesis} that a partial converse holds under a higher connectedness hypothesis.
\end{remark}

\begin{proposition}\label{closedprop} 
If $X$ is first countable and locally path-connected, then $\wildn(X)$ is closed in $X$.
\end{proposition}

\begin{proof}
Suppose $X$ is first countable and locally path-connected and that $x\in \ov{\wildn(X)}$. Let $U_1\supseteq U_2\supseteq U_3\supseteq \cdots$ be a neighborhood base at $x$ of path-connected sets and pick points $x_k\in \wildn(X)\cap U_k$ for each $k\geq 1$. For each $k$, find a fully essential map $f_k:(\bbe_n,b_0)\to (X,x_k)$ with $j$-th restriction $f_{k,j}=f_k\circ \ell_j:S^n\to X$. For each $k$, find $J_k$ such that $\im(f_{k,j})\subseteq U_k$ for all $j\geq J_k$. Let $\alpha_k:\ui\to U_k$ be a path from $x$ to $x_k$ and let $g_k:S^n\to U_k$ be the map based at $x$, which is the path-conjugate of $f_{k,J_k}$ by the path $\alpha_k$. Now $\{g_k\}_{k\in\bbn}$ is a sequence of essential based maps $g_k:(S^n,s_0)\to (X,x)$, which converges to $x$. Thus $x\in\wild_n(X)$, proving that $\wildn(X)=\ov{\wildn(X)}$.
\end{proof}

\begin{corollary}\label{closedcor}
If $X$ is a Peano continuum, then $\wildn(X)$ is a compact metrizable space.
\end{corollary}

In the next two examples, we illustrate that the lack of either hypothesis in Proposition \ref{closedprop} (first countability or local path connectivity) can lead to $\wildn(X)$ failing to be closed in $X$.

\begin{example}[Lack of first countability]\label{nonclosedexample}
Let $\{A_k\}_{k\in\bbn}$ be a sequence of homeomorphic copies of $\bbe_n$ with canonical basepoint $a_k\in A_k$. Let $X=(\ui\sqcup\coprod_{k\geq 1}A_k)/\mathord{\sim}$ be the quotient space obtained by attaching $A_k$ to $\ui$ by $a_k\sim \frac{1}{k}$ (see Figure \ref{fig1} in the case $n=1$). Since $X$ has the weak topology with respect to the subspaces $\ui$ and $A_k$, $k\geq 1$, $X$ is locally path-connected at $0$ but is not first countable at $0$. In particular, any compact set, e.g. the image of a map $\bbe_n\to X$, must have image in a subspace $Y$ of $X$, which is the union of $\ui$ and finitely many $A_k$. But any such subspace $Y$ is locally contractible at $0$. Thus $0$ is not a $\pi_n$-wild point of $X$ and we have that $\wildn(X)=\{1/k\mid k\in\bbn\}$ is not closed in $X$. 
\end{example}

\begin{figure}[H]
\centering \includegraphics[height=1.1in]{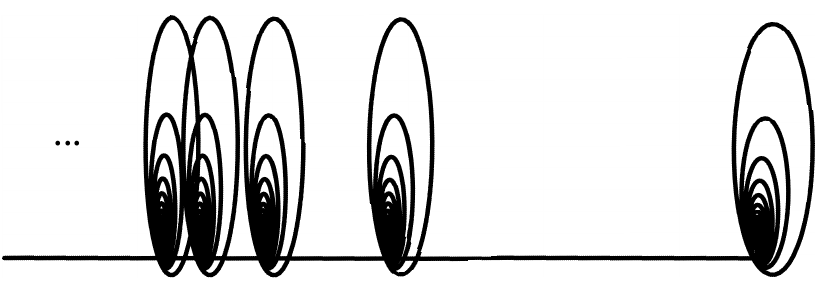}
\caption{\label{fig1} A space obtained by attaching copies of $\bbe_1$ to $\ui$ along the points $1/k$ (in the weak topology).}
\end{figure}

\begin{example}[Lack of local path connectivity]\label{sinecurveexample}
Let $T\subseteq W\subseteq \bbr^2$ where $T$ is the closed topologist sine curve and $W$ is a Warsaw circle containing $T$. Let $A=\{a_1,a_2,a_3,\dots\}$ be a countable dense subset of the non-compact path-component $P_1$ of $T$. Let $X$ be the space obtained by attaching an $n$-sphere of radius $1/k$ at the point $a_k\in P_1$ (see Figure \ref{fig3} in the case $n=1$). Although $X$ is a path-connected compact metric space, it is not locally path-connected at any point in the compact path-component $P_2$ of $T$. Note that $\wildn(X)=P_1$, which is not closed in $X$. In particular, one cannot form a fully essential map $f:\bbe_n\to X$ based at a point of $P_2$ because there are no small paths between $P_1$ and $P_2$ that one can use to form a \textit{shrinking} sequence of path-conjugates (as in the proof of Proposition \ref{closedprop}).
\end{example}

\begin{figure}[H]
\centering \includegraphics[height=2in]{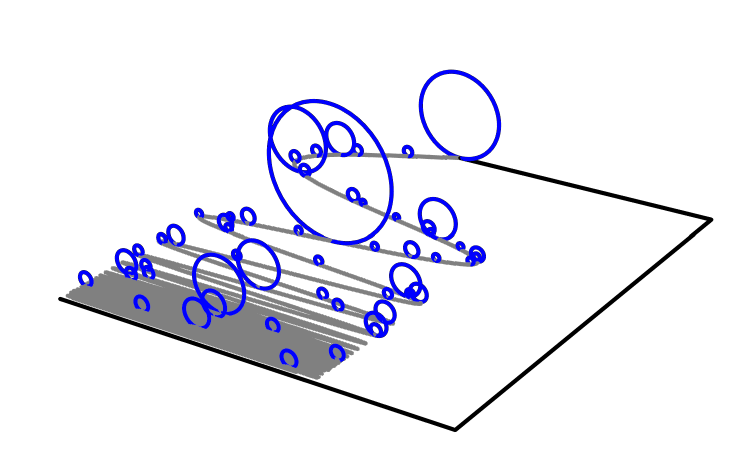}
\caption{\label{fig3}The Warsaw circle in the $xy$-plane with a sequence of circles of shrinking radius attached along a dense subset of the non-compact path component of the topologist's sine curve (illustrated in gray).}
\end{figure}

When dealing with subspaces of real Euclidean space we note the following consequence of dimension theory. When referring to topological dimension $\dim(X)$ of a space $X$ we mean ``Lebesgue covering dimension" (this agrees with small and large inductive dimension when $X$ is a separable metric space).

\begin{proposition}
If $m\geq 2$, $n\geq 0$, and $X\subseteq \bbr^m$, then $\dim(\wildn(X))\leq m-1$.
\end{proposition}

\begin{proof}
We first check that the interior $\int(\wildn(X))$ of $\wildn(X)$ in $\bbr^m$ is empty. If $x\in \int(\wildn(X))$, then there exists an open Euclidean $\epsilon$-ball $U$ such that $x\in U\subseteq \int(\wildn(X))\subseteq X$ and a fully essential map $f:(\bbe_n,b_0)\to (X,x)$. Since $U$ is open there exists $j$ sufficently large so that $f$ maps the $j$-th sphere of $\bbe_n$ into $U$. However, $U$ is contractible and so the $j$-th restriction $f_j:S^n\to X$ is null-homotopic in $X$, which is a contradiction. We conclude that $\int(\wildn(X))=\emptyset$. It is a well-known result of dimension theory \cite[1.8.10]{EngelkingDimThry} that if $M\subseteq \bbr^m$ has topological dimension $\dim(M)=m$, then the interior $\int(M)$ of $M$ in $\bbr^m$ is non-empty. Since $\int(\wildn(X))=\emptyset$, we must have $\dim(\wildn(X))\leq m-1$.
\end{proof}

For example, the $\pi_1$-wild set of a planar set must be $1$-dimensional (planar sets are aspherical so no higher wild sets are non-empty \cite{CCZaspherical}) and the $\pi_n$-wild set of a subset of Euclidean $3$-space can have dimension at most $2$ for all $n\geq 1$.

\begin{definition}
We say that a space $X$ is \textit{perfectly $\pi_n$-wild} if $\wild_n(X)=X$.
\end{definition}

To provide a simple first example of a perfectly $\pi_n$-wild space, we consider a higher-dimensional analogue of the Sierpinski Carpet construction.

\begin{example}\label{cubeexample}
Let $n\geq 0$ and $Q_0=[0,1]^{n+1}$ be the unit $(n+1)$-cube. If $Q_m$ is defined, we let $Q_{m+1}$ be the set of all $(x_1,x_2,\dots,x_n)\in \bbr^n$ such that there exist $(a_1,a_2,\dots,a_n)\in \{0,1,2\}^n$ such that $(3x_i-a_i)_{i}\in Q_m$ and such that not all $a_i$ are equal to $1$. Let $Q_{\infty}=\bigcap_{m\geq 0}Q_m$. 

If $n=0$, then $Q_{\infty}$ is the ternary Cantor set and if $n=1$, then $Q_{\infty}$ is the Sierpinski Carpet. If $n=2$, then $Q_{\infty}$ is not the Menger cube but rather a Peano continuum more analogous to the Sierpinski carpet where one removes the interior of the central $n$-cube $[1/3,2/3]^3$ from $[0,1]^3$ and then recursively removes the interior of the analogous ternary-central $3$-cube from each of the $26$ remaining $3$-cubes that share a face with $[1/3,2/3]^3$ (see Figure \ref{fig4}). In general, $Q_{\infty}$ is an $n$-dimensional Peano continuum such that $[0,1]^{n+1}\backslash Q_{\infty}$ is a disjoint union of countably many open $(n+1)$-cubes (of null diameter). For each connected component $C$ of $[0,1]^{n+1}\backslash Q_{\infty}$, which is an open $(n+1)$-cube, $\partial C$ is a retract of $Q_{\infty}$ and so a given homeomorphism $S^{n}\to \partial C$ is essential in $Q_{\infty}$. Moreover, for any $x\in Q_{\infty}$ and path-connected open neighborhood $U$ of $x$ in $Q_{\infty}$, there is some connected component $C$ of $[0,1]^{n+1}\backslash Q_{\infty}$ such that $\partial C\subseteq U$. It follows that $x\in\wildn(Q_{\infty})$. Thus $Q_{\infty}$ is perfectly $\pi_n$-wild.
\end{example}

\begin{figure}[H]
\centering \includegraphics[height=2in]{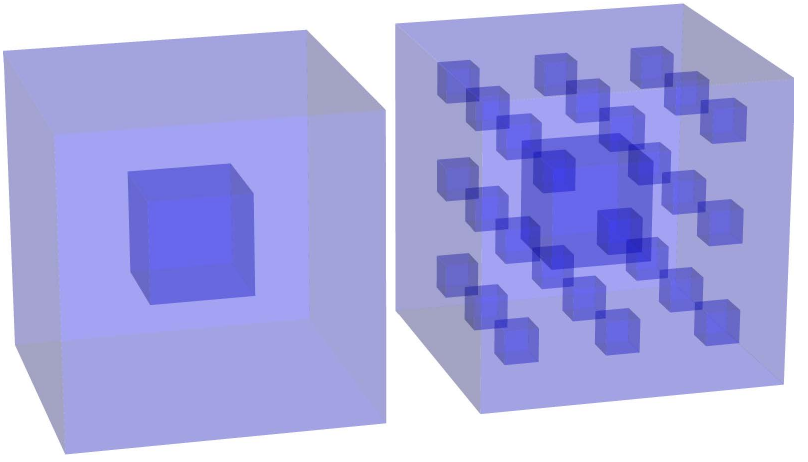}
\caption{\label{fig4}The stages $Q_1$ and $Q_2$ in the construction of the $2$-dimensional case of the Peano continuum $Q_{\infty}$.}
\end{figure}

In dimension $n=2$, the space $Q_{\infty}$ from Example \ref{cubeexample} has the property that every point $x\in Q_{\infty}$ is an accumulation point of subspaces $C_1,C_2,C_3,\dots$, which are homeomorphic to $S^2$ and each of which is a retract of $Q_{\infty}$. Since $Q_{\infty}$ is a Peano continuum and $\pi_k(C_j)\neq 0$ for all $k\geq 2$, we have the following theorem as an observation.

\begin{theorem}\label{perfectlywildthm}
There exists a $2$-dimensional Peano continuum in $\bbr^3$ that is perfectly $\pi_k$-wild for all $k\geq 2$.
\end{theorem}

\section{Basic Properties of $\pi_n$-wild sets}\label{section3properties}

Here, we relate the $\pi_n$-wild set operation to basic operations such as locally path-connected coreflections, coproducts, and direct products. Recall from Example \ref{sinecurveexample} that $\wildn(X)$ need not be closed in $X$ if $X$ is not locally path-connected. The next construction is a well-known method of refining the topology of a space to obtain a locally path-connected space without changing the weak homotopy type (or underlying wild set) of the space. 

\begin{definition}
If $X$ is a space, then the \textit{locally path-connected coreflection} of $X$ is the space $\lpc(X)$ with the same underlying set as $X$ but with the topology generated by the path components of open sets in $X$.
\end{definition}

The space $\lpc(X)$ is locally path-connected. Moreover, identity function $\lpc(X)\to X$ is continuous and universal in the sense that if $f:Z\to X$ is a map from a locally path-connected space $Z$, then $f:Z\to \lpc(Z)$ is also continuous \cite[Theorem 2.2]{BDLM08}. It follows that the identity function $\lpc(X)\to X$ is a weak homotopy equivalence.

\begin{proposition}\label{lpcwildconnectionprop}
For any space $X$ and $n\geq 1$, the identity function $\wildn(\lpc(X))\to\wildn(X)$ is continuous and we have $\lpc(\wildn(\lpc(X)))=\lpc(\wildn(X))$ as spaces.
\end{proposition}

\begin{proof}
Using the universal property of $\lpc(X)$ and the fact that $\bbe_n$, $S^n$, and the closed unit $(n+1)$-disk are locally path-connected, it is straightforward to show that a function $f:\bbe_n\to X$ is continuous (and fully essential) if and only if $f:\bbe_n\to\lpc(X)$ is continuous (and fully essential). Thus $\wildn(\lpc(X))$ and $\wildn(X)$ are equal as subsets of $X$. Since we know the sets $\wildn(\lpc(X))$ and $\wildn(X)$ are equal, the continuous identity function $\lpc(X)\to X$ restricts to the continuous identity function $\wildn(\lpc(X))\to\wildn(X)$.

For the second statement, apply the functor $\lpc$ to the continuous identity map $\wildn(\lpc(X))\to\wildn(X)$ from the first statement to see that the identity function $\lpc(\wildn(\lpc(X)))\to\lpc(\wildn(X))$ is continuous. The inclusion $i:\wildn(X)\to X$ induces a continuous injection $i:\lpc(\wildn(X))\to \lpc(X)$ and we know the image of this map is $\wildn(\lpc(X))$. Hence, the identity function $\lpc(\wildn(X))\to \wildn(\lpc(X))$ is continuous. Applying $\lpc$ to this map gives that the identity function $\lpc(\wildn(X))\to \lpc(\wildn(\lpc(X)))$ is also continuous. We conclude that the identity function $\lpc(\wildn(\lpc(X)))\to\lpc(\wildn(X))$ is a homeomorphism.
\end{proof}

\begin{corollary}
The $\pi_n$-wild sets of $X$ and $\lpc(X)$ are weakly homotopy equivalent by a bijection. Moreover, if $X$ is first countable, then $\wildn(\lpc(X))$ is closed in $\lpc(X)$.
\end{corollary}

\begin{proof}
Recall that for any space $Y$, the identity function $\lpc(Y)\to Y$ is a weak homotopy equivalence. Applying $\lpc$ to the identity function $\wildn(\lpc(X))\to\wildn(X)$ from Proposition \ref{lpcwildconnectionprop} gives the following commutative diagram of identity functions.
\[\xymatrix{
\lpc(\wildn(\lpc(X))) \ar[d] \ar@{=}[r] & \lpc(\wildn(X)) \ar[d]\\
\wildn(\lpc(X)) \ar[r] & \wildn(X) 
}\]Since top map is an identity map of spaces (the second statement of Proposition \ref{lpcwildconnectionprop}) and the vertical maps are weak homotopy equivalences, the bottom map is a bijective weak homotopy equivalence.

If $X$ is first countable, the definition of $\lpc(X)$ ensures that $\lpc(X)$ is also first countable. Proposition \ref{closedprop} then applies to $\lpc(X)$, proving the second statement.
\end{proof}

\begin{remark}\label{lpcwildconnectionexmaple}
Even though $\wildn(\lpc(X))$ has the same underlying set as $\wildn(X)$ and is guaranteed to have a topology that is finer than or equal to that of $\wildn(X)$, the two need not be homeomorphic. For example, if $X=\bbe_0\times \bbe_n$, then $\wildn(X)=\bbe_0\times\{b_0\}\cong \bbe_0$. But $\lpc(X)=disc(\bbe_0)\times\bbe_n$ where $disc(\bbe_0)$ is the underlying set of $\bbe_0$ with the discrete topology. Then $\wildn(\lpc(X))=disc(\bbe_0)\times\{b_0\}\cong disc(\bbe_0)$ is discrete.
\end{remark}

We omit the proof of the following basic proposition.

\begin{proposition}\label{disjointunionprop}
For any $n\geq 0$ and collection of spaces $\{X_{\lambda}\}_{\lambda}$, we have \[\wildn\left(\coprod_{\lambda}X_{\lambda}\right)=\coprod_{\lambda}\wildn(X_{\lambda}).\]
\end{proposition}

Infinite direct products provide an abundance of examples of perfectly $\pi_n$-wild spaces. We characterize their $\pi_n$-wild sets in the next proposition. 

\begin{proposition}\label{productprop}
Let $\{X_i\}_{i\in I}$ be a family of path-connected spaces with direct product $X=\prod_{i\in I}X_i$.
\begin{enumerate}
\item If $I$ is finite, then $X\backslash\wildn(X)=\prod_{i\in I}(X_i\backslash \wildn(X_i))$.
\item If $I$ is infinite and $\pi_n(X_i)$ is trivial for all but finitely many $i\in I$, then $X\backslash\wildn(X)=\prod_{i\in I}(X_i\backslash \wildn(X_i))$.
\item If for infinitely many $i\in I$, $X_i$ has non-trivial $n$-th homotopy group, then $X$ is perfectly $\pi_n$-wild.
\end{enumerate}
\end{proposition}

\begin{proof}
(1) Let $p_i:X\to X_i$, $i\in I$ denote the projection maps and fix a point $(x_i)\in X$. If $(x_i)\in \wildn(X)$, there is a fully essential map $f:(\bbe_n,b_0)\to (X,(x_i))$. Let $f_j=f\circ\ell_j:S^n\to X$ denote the $j$-th restriction of $f$. Then for each $j\in \bbn$, there exists $i_j\in I$ such that $\pi_{i_j}\circ f_j:(S^n,s_0)\to( X_i,x_i)$ is essential. There exists some $i_0$ for which $i_0=i_j$ for infinitely many $j\in \bbn$. This shows that $x_{i_0}\in\wildn(X_{i_0})$. Conversely, suppose there exists for some $i_0\in I$ such that $x_{i_0}\in \wildn(X_{i_0})$. Find a fully essential map $g_{i_0}:(\bbe_n,b_0)\to (X_{i_0},x_{i_0})$ and if $i\neq i_0$, let $g_i:\bbe_n\to X_i$ be the constant map at $x_i$. Then the map $g:\bbe_n\to X$ with $p_i\circ g=g_i$ for all $i\in\bbn$ is fully essential, proving $(x_i)\in \wildn(X)$. 

(2) If $\pi_n(X_i)$ is trivial for all $i\in I$, then $\pi_n(X)$ is trivial and $\wildn(X)=\emptyset$. Otherwise, we may rearrange the product into a finite product where all factors except one have non-trivial $n$-th homotopy group and apply (1).

(3) Let $(x_i)\in X$. Find a countably infinite subset $\{i_1,i_2,i_3,\dots\}\subseteq I$ such that if $j\in\bbn$, then $X_{i_j}$ is path-connected and $\pi_n(X_{i_j},x_{i_j})\neq 0$. For each $j\in\bbn$, find a map $f_{j}:S^n\to X_{i_j}$ based at $x_{i_j}$ that is not null-homotopic. For each $j\in\bbn$, let $g_j:S^n\to X$ be the map whose $i_j$-th projection is $f_{j}$ and where all other projections are constant at $x_i$. Define $g:(\bbe_n,b_0)\to (X,(x_i))$ so that the restriction of $g$ to the $j$-th sphere is $g_j$. Because $X$ has the product topology and all projections of $g$ are continuous, $g$ is continuous. Moreover, the restriction of $g$ to the $j$-th sphere is not null-homotopic in $X$ and thus $(x_i)\in \wild_n(X)$.
\end{proof}

\begin{example}\label{prodcor}
For binary products, we have $\wildn(X\times Y)=\wildn(X)\times Y\cup X\times\wildn(Y)$. If $X$ and $Y$ are path-connected and both $\pi_n$-wild sets are non-empty, then $\wildn(X\times Y)$ is path-connected. If $\wildn(X)=\{x\}$ and $\wildn(Y)=\{y\}$, then $\wildn(X\times Y)=\{x\}\times Y\cup X\times\{y\}\cong X\vee Y$. Specifically, we have $\wildn(\bbe_n\times\bbe_n)\cong \bbe_n\vee\bbe_n\cong\bbe_n$.
\end{example}

\begin{example}
The infinite dimensional torus $\prod_{i\in\bbn}S^1$ is perfectly $\pi_1$-wild and aspherical. When $k\geq 2$, $\prod_{i\in\bbn}S^k$ is perfectly $\pi_n$-wild whenever $\pi_n(S^k)\neq 0$. On the other hand, since $\pi_n(S^k)=0$ when $n<k$, (2) of Proposition \ref{productprop} gives $\wild_n(S^1\times S^2\times S^3\times\cdots)=\emptyset$ for all $n\geq 0$.
\end{example}

\section{Homotopy invariance of $\pi_n$-wild sets}\label{section4homotopyinvariance}

In general, it is not true that $\wildn(A)\subseteq \wildn(X)$ whenever $A$ is a subspace of $X$. For example, $\bbe_{n}\subseteq \bbr^{n+1}$ where $\wildn(\bbe_n)=\{b_0\}$ and $\wildn(\bbr^{n+1})=\emptyset$.

\begin{definition}
We say a map $f:X\to Y$ is \textit{$\pi_n$-injective} if the induced homomorphism $f_{\#}:\pi_n(X,x)\to \pi_n(Y,f(x))$ is injective for every $x\in X$ (note that $f_{\#}$ is a function if $n=0$).
\end{definition}

\begin{lemma}\label{inductionsteplemma}
If $f:X\to Y$ is $\pi_n$-injective, then $f(\wild_{n}(X))\subseteq \wild_{n}(Y)$. Moreover, any (free) homotopy $H:X\times \ui\to Y$ between $\pi_n$-injective maps $f,g:X\to Y$, restricts to a homotopy $G:\wild_{n}(X)\times \ui\to \wild_n(Y)$ between maps  $f|_{\wild_{n}(X)},g|_{\wild_{n}(X)}:\wild_{n}(X)\to \wild_{n}(Y)$.
\end{lemma}

\begin{proof}
If $x\in \wild_n(X)$, then there is a fully essential map $\alpha:(\bbe_n,b_0)\to (X,x)$. Since $f$ is $\pi_n$-injective, $f\circ\alpha$ is fully essential. Thus $f(x)\in\wildn(Y)$, proving $f(\wild_{n}(X))\subseteq \wild_{n}(Y)$. For the second statement, suppose $H:X\times \ui\to Y$ is a map such that $H(x,0)=f(x)$ and $H(x,1)=g(x)$. Recall from Example \ref{prodcor} that $\wild_n(X\times \ui)=\wild_n(X)\times \ui$. Since $H\circ i=f$ where the inclusion $i:X\to X\times \ui$, $i(x)=(x,0)$ is a homotopy equivalence, $H$ is $\pi_n$-injective. Therefore $H(\wild_n(X)\times \ui)=H(\wild_n(X\times \ui))\subseteq \wild_n(Y)$. If $G:\wild_n(X)\times \ui\to \wild_n(Y)$ is the restriction of $H$ to $\wild_n(X)\times \ui$, then $G$ is a homotopy from $f|_{\wild_{n}(X)}$ to $g|_{\wild_{n}(X)}$.
\end{proof}

\begin{corollary}\label{retractcor}
If $n\geq 1$ and $A\subseteq X$ is a retract, then $\wild_n(A)\subseteq \wild_n(X)$.
\end{corollary}


\begin{corollary}\label{wedgecor}
Suppose $X\vee Y$ has wedgepoint $x_0$. Then \[\wildn(X)\cup\wildn(Y)\subseteq \wildn(X\vee Y)\subseteq \wildn(X)\cup\wildn(Y)\cup\{x_0\}.\]
\end{corollary}

\begin{proof}
Since $X$ and $Y$ are retracts of $X\vee Y$, we have $\wildn(X)\cup\wildn(Y)\subseteq \wildn(X\vee Y)$ by Corollary \ref{retractcor}. For the second inclusion, suppose $x\in \wildn(X\vee Y)\backslash\{x_0\}$. If $x\in X\backslash\{x_0\}$, then there is a fully essential map $f:(\bbe_n,b_0)\to (X\vee Y,x)$. Since $X\backslash\{x_0\}$ is open in $X\vee Y$, we may assume $\im(f)\subseteq X\backslash\{x_0\}$. If the $j$-th restriction of $f$ is inessential in $X$, then it is inessential in $X\vee Y$. Thus $f:\bbe_n\to X$ is fully essential and we have $x\in\wildn(X)$. Similarly, if $x\in Y\backslash\{x_0\}$, then the same argument gives $x\in \wildn(Y)$. This proves $\wildn(X\vee Y)\backslash\{x_0\} \subseteq\wildn(X)\cup\wildn(Y)$, which implies the second inclusion.
\end{proof}

\begin{example}\label{wedgeexample}
In general, it is not true that $\wildn(X\vee Y)=\wildn(X)\cup \wildn(Y)$. For example if $C\bbe_1=\bbe_1\times\ui/\bbe_1\times\{1\}$ is the cone over the $1$-dimensional earring space where the basepoint $x_0$ is the image of $(b_0,0)$, then $C\bbe_1\vee C\bbe_1$ is the well-known Griffiths double cone \cite{CorsonCone,EdaFischer,Griffiths}. Since $C\bbe_1$ is contractible, we have $\wild_1(C\bbe_1)=\emptyset$. However, $\wild_1(C\bbe_1\vee C\bbe_1)=\{x_0\}$. In contrast, $\wild_n(C\bbe_n\vee C\bbe_n)=\emptyset$ when $n\geq 2$ since $\pi_n(C\bbe_n\vee C\bbe_n)=0$ \cite{EK00higher}. However, the authors suspect that $\wild_{2m-1}(C\bbe_m\vee C\bbe_m)$ is non-empty for $m\geq 2$ (due to infinite products of Whitehead products) although this appears to be unconfirmed at this point.
\end{example}

\begin{theorem}[homotopy invariance]\label{homotopyinvariance}
For all $n\geq 0$, the homotopy type of $\wild_n(X)$ is a homotopy invariant of $X$.
\end{theorem}

\begin{proof}
Let $f:X\to Y$ and $g:Y\to X$ be homotopy inverses with homotopies $H:X\times \ui \to X$ from $id_{X}$ to $g\circ f$ and $G:Y\times \ui \to Y$ from $id_{Y}$ to $f\circ g$. Fix $n\geq 0$. Since $f$ and $g$ are $\pi_n$-injective, we have $f(\wild_n(X))\subseteq \wild_n(Y)$ and $g(\wild_n(Y))\subseteq \wild_n(X)$. By the second statement of Lemma \ref{inductionsteplemma}, $H$ restricts to a homotopy $H':\wild_{n}(X)\times \ui \to \wild_{n}(X)$ from $id_{\wild_{n}(X)}$ to $f\circ g|_{\wild_{n}(X)}$. Similarly, $G$ restricts to a homotopy $G':\wild_{n}(Y)\times \ui \to \wild_{n}(Y)$ from $id_{\wild_{n}(Y)}$ to $g\circ f|_{\wild_{n}(Y)}$. Thus $f|_{\wild_{n}(X)}:\wild_{n}(X)\to \wild_{n}(Y)$ and $g|_{\wild_{n}(Y)}:\wild_{n}(Y)\to \wild_{n}(X)$ are homotopy inverses.
\end{proof}

\begin{corollary}
If $\wildn(X)\neq \emptyset$ for some $n\geq 0$, then $X$ is not homotopy equivalent to a CW-complex or a manifold.
\end{corollary}

Since two totally path-disconnected homotopy equivalent spaces must be homeomorphic, we have the following.

\begin{corollary}\label{tpdcor}
If $X\simeq Y$ and $\wildn(X)$ and $\wildn(Y)$ are totally path-disconnected, then $\wildn(X)\cong\wildn(Y)$.
\end{corollary}

\begin{example}
Suppose $X$ and $Y$ are spaces with finitely many $\pi_n$-wild points. If $\wildn(X)$ and $\wildn(Y)$ have a distinct number of elements, then Corollary \ref{tpdcor} implies that $X\nsimeq Y$. Specifically, suppose $T$ is a tree and $k,m$ are distinct natural numbers. If $X$ is obtained by attaching $k$ copies of $\bbe_n$ to $k$-distinct points in $T$ and $Y$ is obtained by attaching $m$-copies of $\bbe_n$ to $m$-distinct points in $T$, then the Hurewicz Theorem and Mayer-Vietoris Sequence apply to show that $X$ and $Y$ are both $(n-1)$-connected and have $n$-th homotopy group isomorphic to $\bbz^{\bbn}$. However, $X\nsimeq Y$ since $X$ and $Y$ have a distinct finite number of $\pi_n$-wild points.

Although it appears that $X$ and $Y$ have isomorphic homotopy groups, it is unlikely to provide a counterexample to Problem \ref{whproblem}. For, if $f:X\to Y$ is a weak homotopy equivalence, Lemma \ref{inductionsteplemma} implies that $f(\wildn(X))\subseteq \wildn(Y)$. If $k<m$, then it is easy to show that $f_{\#}:\pi_n(X,x_0)\to\pi_n(Y,y_0)$ cannot be surjective. If $k>m$, then $f$ must identify two wild points and one should be able to use infinite products of Whitehead products to show that $f_{\#}:\pi_{2n-1}(X,x_0)\to\pi_{2n-1}(Y,y_0)$ is not surjective. At this point, this last claim is conjectural and a proof is likely to require a complete description of $\pi_{2n-1}(X,x_0)$.
\end{example}

\begin{example}
We can also distinguish homotopy types if we modify the previous example by attaching infinite earrings of different dimensions. Suppose $n>m\geq 2$ such that $\pi_n(\bbe_m)\neq 0$. Let $X$ be the space obtained by attaching a copy of $\bbe_m$ and $\bbe_n$ to $\ui$ by identifying the respective wedgepoints with $0$ and $1$. We compare this space with $\bbe_m\vee \bbe_n$. Then $\wild_m(X)$ and $\wild_m(\bbe_m\vee \bbe_n)$ both contain a single point since $\wild_m(\bbe_n)=\emptyset$. However, $\wild_n(X)=\{0,1\}$ while $\wild_n(\bbe_m\vee \bbe_n)$ contains a single point. Thus $X\nsimeq \bbe_m\vee \bbe_n$. In particular, the quotient map $X\to \bbe_m\vee \bbe_n$ collapsing the arc to a point is not a homotopy equivalence.
\end{example}

\begin{example}\label{wildcircleexample}
For $n\geq 1$, consider an $n$-dimensional Peano continuum $WS^n$ obtained by attaching a sequence of $n$-spheres whose diameters approach $0$ to $S^n$ along the points of an enumerated dense subset of $S^n$ (a topological version of this construction will be formalized in the next section). We refer to $WS^n$ as the ``wild $n$-sphere" or the ``wild circle" in the case $n=1$ (see left image in Figure \ref{fig2}). Then $\wild_n(WS^n)=S^n$ is not $n$-connected. Moreover, according to Theorem \ref{homotopyinvariance}, $WS^n$ cannot be homotopy equivalent to any space $Y$ where $\wild_1(Y)$ is not homotopy equivalent to $S^n$. For instance, if $m\neq n$, then $WS^m\nsimeq WS^n$.
\end{example}

\begin{example}[A planar set not homotopy equivalent to any one-dimensional space]
It is shown in \cite{CCZaspherical} that there exists planar Peano continua, which are not homotopy equivalent to any one-dimensional Peano continuum. Here, we give a simple example and elementary argument using the homotopy invariance of wild sets. Let $WS^1$ be the wild circle from Example \ref{wildcircle} and let $Z=WS^1\cup\bbd^{2}$ where $\bbd^2$ is the closed unit disk (see the right image in Figure \ref{fig2}). We still have $\wild_{1}(Z)=S^1$ but the inclusion $j:S^1\to Z$ is null-homotopic. Suppose $Z'$ is a one-dimensional space and $f:Z\to Z'$ and $g:Z'\to Z$ are homotopy inverses. Then $f$ and $g$ restrict to a homotopy equivalence $S^1\simeq \wild_1(Z')$ on wild $\pi_1$-sets. However, every inclusion map of one-dimensional spaces is $\pi_1$-injective \cite[Corollary 3.3]{CConedim} and so the inclusion $k:\wild_1(Z')\to Z'$ is $\pi_1$-injective. Since $f|_{S^1}=f\circ j$ is not-null-homotopic in $\wild_1(Z')$, $k\circ f|_{S^1}$ is not null-homotopic in $Z'$; a contradiction. 

The above argument actually implies that any space $X$ for which the inclusion $\wild_1(X)\to X$ is not $\pi_1$-injective cannot be homotopy equivalent to a one-dimensional space. On the other hand, when $n\geq 2$, an inclusion map $A\to X$ of $n$-dimensional metric spaces need not be $\pi_n$-injective, e.g. $S^1\vee S^n\to S^n\vee S^n$. Hence, the argument does not extend to higher dimensions.
\end{example}

\begin{figure}[ht]
\centering \includegraphics[height=2.4in]{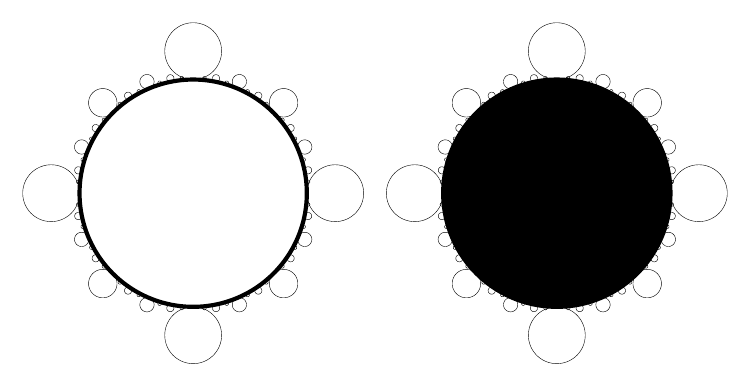}\label{fig2}
\caption{\label{wildcircle}A one-dimensional Peano continuum with a non-simply connected $\pi_1$-wild set (left) and the filled-in version (right), which is not homotopy equivalent to any one-dimensional space.}
\end{figure}

\begin{example}
Let $X$ be obtained by attaching a sequence $A_1,A_2,A_3,\dots$ of copies of $S^n$ with diameters approaching $0$ along a dense set in $\ui$ (see Figure \ref{fig6}). Then $\wildn(X)=\ui$. Since $\ui$ is homotopy equivalent to $\wildn(\bbe_n)=\{b_0\}$, it does not follow directly from Theorem \ref{homotopyinvariance} that the homotopy type of $X$ is distinct from $\bbe_n$. This is a motivation for our rigidity result (Theorem \ref{thm2}), which applies in this case to distinguish the homotopy types of $X$ and $\bbe_n$.
\end{example}

\begin{figure}[ht]
\centering \includegraphics[height=1.2in]{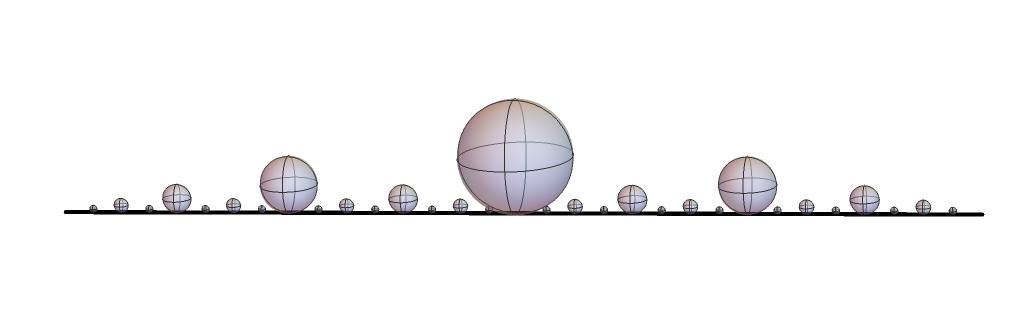}
\caption{\label{fig6}Attaching a shrinking sequence of $2$-spheres along the dyadic rationals in $\ui$.}
\end{figure}




\section{Constructing Spaces with Prescribed Wild Sets}\label{section5makingspaceswild}

In this section, our main goal is to prove Theorem \ref{thm1}, which implies that every compact metric space $X$ may be realized as the $\pi_n$-wild set of some Peano continuum. Our construction occurs in two steps. First, we attach a countable sequence of shrinking arcs to $X$ to obtain a space that is guaranteed to be a Peano continuum. Second, we attach a sequence of shrinking copies of $\bbe_n$ along a dense set in $X$ (not affecting the arcs attached in the first step) to ensure the resulting space is wild at all points of $X$.



\begin{lemma}\label{peanocontlemma}
Let $n\geq 1$. For every compact metric space $X$ there exists a Peano continuum $Y$ such that
\begin{enumerate}
\item $X\subseteq Y$ and $Y\backslash X$ is a disjoint union of countably many open arcs,
\item the inclusion $i:X\to Y$ is $\pi_1$-injective,
\item $\wild_n(X)\subseteq \wild_n(Y)\subseteq X$,
\item $\dim(Y)= \max\{1,\dim(X)\}$.
\end{enumerate}
\end{lemma}

\begin{proof}
If $\dim(X)=0$, we can identify $X$ with a compact subset of $\bbr$. Let $a=\min(X)$ and $b=\max(X)$ in $\bbr$ and set $Y=[a,b]$ to satisfy the conditions of the theorem. We now assume $\dim(X)\geq 1$. Let $\mcc\subseteq \ui$ be the Ternary Cantor Set and let $f:\mcc\to \ui$ be the inclusion. By the Hausdorff-Alexandroff Theorem \cite[\S 7.3, Theorem 7.7]{Nadler}, there exists a continuous surjection $g:\mcc\to X$. Let $Y$ be the pushout of $f$ and $g$, that is the quotient space $X\sqcup \ui/\mathord{\sim}$ where $f(c)\sim g(c)$ for all $c\in \mcc$. Let $Q:X\sqcup\ui\to Y$ be the quotient map. Since $f$ is injective, the induced map $i:X\to Y$ is injective and since $g$ is surjective, the induced map $q:\ui\to Y$ is surjective (using basic properties of pushouts).

First, we show that $Y$ is Hausdorff. Suppose $y_1,y_2\in Y$ are disjoint points. It is straightforward to check that $q$ maps $\ui\backslash \mcc$ homeomorphically onto $Y\backslash i(X)$. Thus we may focus our attention to the case where at least one of $y_1$ or $y_2$ lies in $i(X)$. Let $\{(a_j,b_j)\mid j\in \bbn\}$ be an enumeration of the connected components of $\ui\backslash\mcc$ and let $m_j=\frac{a_j+b_j}{2}$ be the midpoint. If $y_1\in i(X)$ and $y_2\in q((a_j,b_j))$ for some $j\in\bbn$, find $a_j<s<t<r<b_j$ where $q(t)=y_2$. Now $Y\backslash q([s,r])$ and $q((s,r))$ are disjoint open neighborhoods in $Y$ of $y_1$ and $y_2$ respectively. Suppose $y_1,y_2\in i(X)$. Given an open set $U\subseteq X$, let
\begin{itemize}
\item[] $C_{U}=\{j\in J\mid g(\{a_j,b_j\})\subseteq U\}$,
\item[] $L_U=\{j\in J\mid g(a_j)\in U\text{ and }g(b_j)\notin U\}$,
\item[] $R_U=\{j\in J\mid g(a_j)\notin U\text{ and }g(b_j)\in U\}$.
\end{itemize}
Define \[E(U)=U\cap \bigcup_{j\in C_U}q((a_j,b_j))\cup \bigcup_{j\in L_U}q((a_j,m_j))\cup \bigcup_{j\in R_U}q((m_j,b_j)).\]
Note that $Q^{-1}(E(U))$ is open in $X\sqcup\ui$ and thus $E(U)$ is open in $Y$. Find disjoint open neighborhoods $U,V$ in $X$ containing $y_1,y_2$ respectively. Then $E(U)$ and $E(V)$ are disjoint neighborhoods of $y_1$ and $y_2$ in $Y$, completing the proof that $Y$ is Hausdorff. By the Hahn-Mazurkiewicz Theorem \cite{Nadler}, the continuous image of $\ui$ onto a Hausdorff space is a Peano continuum. Thus $Y$ is a Peano continuum. It follows that $q:\ui\to Y$ is a quotient map and $i:X\to Y$ is an embedding. Thus, we may identify $X$ naturally as a subspace of $Y$. Since $\ui\backslash \mcc$ is a disjoint union of countably many open $1$-cells, so is $Y\backslash X$. 

For (2), we note that since $Y\backslash X$ is a disjoint union of open $1$-cells, Lemma 4.3 of \cite{CMRZZ08} implies that the inclusion $i:X\to Y$ is $\pi_1$-injective.

(3) follows from (2),  Lemma \ref{inductionsteplemma}, and the fact that $Y$ is locally contractible at the points of $Y\backslash X$.

For (4), recall that we have assumed $\dim(X)\geq 1$. Since $X$ embeds in $Y$, we have $\dim(X)\leq \dim(Y)$. That $\dim(Y)\geq \dim(X)$ follows from the ``Sum Theorem" in dimension theory \cite[1.5.3]{EngelkingDimThry}.
\end{proof}

\begin{remark}
Although we do not prove it here, it follows from forthcoming work of the first author and Curtis Kent on generalized covering spaces in the sense of Fischer-Zastrow \cite{FZ07} that the inclusion map $i:X\to Y$ in Lemma \ref{peanocontlemma} is, in fact, $\pi_{m}$-injective for all $m\geq 1$.
\end{remark}

\begin{example}\label{uncountableexample}
Let $n\geq 2$. The space $X=\bbe_1\vee S^n$ is an $n$-dimensional Peano continuum whose $n$-th homotopy group is isomorphic to the uncountable free-abelian group $\bbz[\pi_1(\bbe_1)]$ \cite[Example 7.4]{AcetiBrazas}. However, if $U$ is a contractible neighborhood of the basepoint in $S^n$, then $\bbe_1\vee U$ is deformation retracts onto $\bbe_1$, which is aspherical \cite{CFhigher}. It follows that $X$ has a neighborhood base of aspherical sets at the wedgepoint. Hence, $\wild_n(X)=\emptyset$ even though $\pi_n(X)$ is uncountable.
\end{example}

We also have the following geometric version of Lemma \ref{peanocontlemma}, which is motivated by Problem \ref{whproblem} and the fact that finite dimensional Peano continua embed into finite dimensional real space.

\begin{corollary}\label{peanocontlemmarn}
For every compact metric space $X\subseteq \bbr^n$ there exists a Peano continuum $Z$ such that $X\subseteq Z\subseteq \bbr^n$ and $Z\backslash X$ is empty or a disjoint union of countably many open line segments.
\end{corollary}

\begin{proof}
As in the proof of the previous Lemma, let $\mcc\subseteq \ui$ be the Ternary Cantor Set, $f:\mcc\to \ui$ be the inclusion, $g:\mcc\to X$ be a continuous surjection, and let $Y$ be the pushout of $f$ and $g$. For each connected component $(a,b)$ of $\ui\backslash \mcc$, let $L(a,b)\subseteq \bbr^n$ be the line segment with endpoints $g(a)$ and $g(b)$. Let $Z$ be the union of $X$ and the line segments $L(a,b)$ ranging over the connected components of $\ui\backslash\mcc$. Let $j:X\to Z$ denote the inclusion map. We may extend $j\circ g:\mcc\to Z$ to a path $\alpha:\ui\to Z$ so that the restriction of $\alpha$ to $[a,b]$ is a linear parameterization of $L(a,b)$. The continuity of $\alpha$ is straightforward to verify using the continuity of $g$ and the fact that linear paths are geodesics. Since $Y$ is a pushout by construction, the maps $\alpha$ and $j$ uniquely induce a surjective map $h:Y\to Z$. Since $Y$ is a Peano continuum so is $h(Y)=Z$. Note that for each component $(a,b)$ of $\ui\backslash \mcc$, $L(a,b)\backslash X$ is empty or a countable disjoint union of line segments. Hence, $Z\backslash X$ is empty or a countable disjoint union of open line segments.
\end{proof}

Fix a space $X$ and a non-empty subspace $A\subseteq X$. We construct a space $Y$ from this pair so that $X\subseteq Y$ and $\wildn(Y)=A$.

\begin{definition}[Shrinking Point-Attachment Spaces]\label{shrinkingattachmentdef}
Let $X$ be a compact space, $A=\{a_j\}_{j\in\bbn}$ be a sequence (of not necessarily distinct points) in $X$ and let $\scrb=\{(B_j,b_j)\}_{j\in\bbn}$ be a sequence of based spaces. Let $\san(X,A,\scrb)=X\sqcup \coprod_{j\in\bbn}B_j/\mathord{\sim}$ where $a_j\sim b_j$ for all $j\in\bbn$, that is, $\san(X,A,\scrb)$ is obtained by attaching each $B_j$ to $X$ by identifying the basepoint of $B_j$ with $a_j$. We give $\san(X,A,\scrb)$ the following topology: $U\subseteq \san(X,A,\scrb)$ is open if and only if
\begin{enumerate}
\item $X\cap U$ is open in $X$
\item $ B_j\cap U$ is open in $B_j$ for all $j\in \bbn$,
\item whenever $x\in X\cap U$ and $j_1<j_2<j_3<\cdots$ is such that $\{a_{j_i}\}_{i\in\bbn}\to x$ in $X$, then $B_{j_i}\subseteq U$ for all but finitely many $i\in\bbn$.
\end{enumerate}
When $\scrb=\{B,B,B,\dots\}$ is constant, we write $\san(X,A,B)$ for the space $\san(X,A,\scrb)$. In general, we will refer to spaces of the form $\san(X,A,\scrb)$ as \textit{shrinking point-attachment spaces}. In the case where $\scrb=\{(\bbe_n,b_0),(\bbe_n,b_0),(\bbe_n,b_0),\dots\}$, we call $\san(X,A,\bbe_n)$ the \textit{$\pi_n$-wildification of $X$ at $A$}.
\end{definition}

It is straightforward to check that Conditions (1)-(3) in the previous definition do, in fact, define a topology. Typically we will identify the sets $X$ and $B_j$ with their images in $\san(X,A,\scrb)$. Moreover, Conditions (1) and (2) mean precisely that the topology of $\san(X,A,\scrb)$ is coarser than the usual weak topology with respect to the subsets $X,B_1,B_2,B_3,\dots$ each with their given topology.

\begin{example}
If $X=\{x_0\}$ contains a single point, then $\scrs(X,A,\scrb)=\sw_{j\in\bbn}B_j$.
\end{example}

For the remainder of this subsection, we use the notation $X$, $A=\{a_j\}_{j\in\bbn}$, $\scrb=\{(B_j,b_j)\}_{j\in\bbn}$, exactly as we do in Definition \ref{shrinkingattachmentdef}. Typically, $A$ will be a sequence of pairwise-distinct points and $\scrb$ will be a constant sequence. When this occurs, the resulting space is independent of the enumeration of $A$ and so we may abuse notation write $A$ to denote the set $\{a_j\in X\mid j\in\bbn\}$.

\begin{proposition}\label{projectionprop}
Let $Z_m$ be the space obtained by attaching $B_1,B_2,\dots,B_m$ to $X$ by identifying $a_j\sim b_j$ (with the usual weak topology). Then the map $\phi_m:\san(X,A,\scrb)\to Z_m$ collapsing $B_j$ to $a_j$ for all $j>m$ is a continuous retraction.
\end{proposition}

\begin{proof}
The inclusion function $Z_m\to \san(X,A,\scrb)$ is continuous by Conditions (1) and (2) defining the topology of $\san(X,A,\scrb)$. We check that $\phi_m$ is continuous. Let $V\subseteq Z_m$ be open. Then \[U=\phi_{m}^{-1}(V)=(V\cap X)\cup \bigcup\{B_j\mid a_j\in V\text{ and }j>m\}\cup \bigcup\{V\cap B_j\mid 1\leq j\leq m\}.\] Since $Z_m$ has the weak topology, it is clear that $U$ satisfies Conditions (1) and (2). Suppose $x\in U$ and $j_1<j_2<j_3<\cdots$ is such that $\{a_{j_i}\}_{i\in\bbn}\to x$ in $X$. Then there exists $i_0$ such that $a_{j_i}\in U\cap X=V\cap X$ for all $i\geq i_0$. By our description of $U$ above, it follows that $B_{j_i}\subseteq U$ for all but finitely many $i\in\bbn$. Thus $U$ satisfies Condition (3) and we conclude that $U$ is open in $\san(X,A,\scrb)$.
\end{proof}

\begin{remark}\label{inverselimitremark}
Let $Z_m$ be defined as in Proposition \ref{projectionprop}. For each $m\in \bbn$, there is a map $\phi_{m+1,m}:Z_{m+1}\to Z_m$, which collapses $B_{m+1}$ to $a_{m+1}$. Let $\varprojlim_{m}(Z_m,\phi_{m+1,m})$ be the inverse limit space, denoted more succinctly as $\varprojlim_{m}Z_m$. The maps $\psi_m:\san(X,A,\scrb)\to Z_m$ from Proposition \ref{projectionprop} agree with the bonding maps $\phi_{m+1,m}$ and induce a continuous bijection $\psi:\san(X,A,\scrb)\to \varprojlim_{m}Z_m$ given by $\psi(x)=(\phi_m(x))_{m\in\bbn}$.
\end{remark}

Note that the construction of $\san(X,A,\scrb)$ is only intended to be useful when $X$ is compact since if a sequence $a_{j_1},a_{j_2},a_{j_3},\dots$ does not have a convergent subsequence then all of the corresponding attached spaces $B_{j_i}$ will be ``large." However, this construction does allow us to attach spaces in a shrinking fashion without appealing to a uniform structure such as a metric.

\begin{proposition}\label{compactprop}
If $X$ and each $B_j\in\scrb$ is compact, then so is $\san(X,A,\scrb)$.
\end{proposition}

\begin{proof}
Let $\scru$ be an open cover of $\san(X,A,\scrb)$. Since $X$ is compact, find $U_1,U_2,\dots,U_r\in\scru$ such that $X\subseteq U=\bigcup_{i=1}^{r}U_i$. Since each $B_j$ is compact, it suffices to show that all but finitely many $B_j$ lie in $U$. Suppose that $j_1<j_2<j_3<\cdots$ are such that $B_{j_i}\nsubseteq U$. Since $X$ is compact, we may replace $\{j_i\}$ with a subsequence so that $\{a_{j_i}\}_{j\in\bbn}$ converges to a point $x\in X$. But Condition (3) in Definition \ref{shrinkingattachmentdef} then implies that $B_{j_i}\subseteq U$ for sufficiently large $i$; a contradiction.
\end{proof}

Since all spaces are assumed to be Hausdorff, Remark \ref{inverselimitremark} and Proposition \ref{compactprop} combine to give the following.

\begin{corollary}\label{invlimitcor}
If $X$ and each $B_j\in\scrb$ is compact and $Z_m$ is defined as in Proposition \ref{projectionprop}, then the induced map $\phi:\san(X,A,\scrb)\to \varprojlim_{m}Z_m$ is a homeomorphism.
\end{corollary}

\begin{proposition}\label{propertiesprop}
If $X$ and each $B_j$ is separable (resp. path-connected, path-connected and locally path-connected), then so is $\san(X,A,\scrb)$.
\end{proposition}

\begin{proof}
If $X$ and each $B_j$ are separable, then the coproduct $X\sqcup \coprod_{j\in\bbn}B_j$ is separable. Since the topology of $\san(X,A,\scrb)$ is coarser than the weak topology, it is the continuous image of $X\sqcup \coprod_{j\in\bbn}B_j$ and is therefore separable.

If $X$ and each $B_j$ are path-connected, it is clear that $\san(X,A,\scrb)$ is path-connected. Lastly, suppose $X$ and each $B_j$ are both path-connected and locally path-connected. As noted, $\san(X,A,\scrb)$ is path-connected. Since $B_j\backslash\{b_j\}$ is locally path-connected and open in $\san(X,A,\scrb)$, it suffices to check that $\san(X,A,\scrb)$ is locally path-connected at each point in $X$. Let $x\in X$ and $U$ be an open neighborhood of $x$ in $\san(X,A,\scrb)$. Let $U_0=X\cap U$ and $U_j=U\cap B_j$ for $j\in\bbn$. Find a path-connected neighborhood $V_0$ of $x$ in $X$ such that $V_0\subseteq U_0$. Let $J=\{j\in\bbn\mid a_j\in V_0\}$. If $j\in J$ and $U_j=B_j$, set $V_j=B_j$. If $j\in J$ and $U_j\neq B_j$, find a path-connected neighborhood $V_j$ of $a_j$ in $B_j$ such that $V_j\subseteq U_j.$ Define $V=V_0\cup \bigcup_{j\in J}V_j$. Certainly, $V$ is path-connected and $V\subseteq U$. It suffices to check that $V$ is open in $\san(X,A,\scrb)$. Conditions (1) and (2) of Definition \ref{shrinkingattachmentdef} are met. We check Condition (3). Suppose $v\in V\cap X$ and $k_1<k_2<k_3<\cdots $ are integers such that $\{a_{k_i}\}_{i\in\bbn}$ converges to $v$. Since $v\in U$ and $U$ is open, we have $B_{k_i}\subseteq U$ for all but finitely many $i$. When $B_{k_i}=U$, we have $V_{k_i}=U_{k_i}=B_{k_i}$. Thus $B_{k_i}\subseteq V$ for all but finitely many $i$.
\end{proof}

\begin{proposition}\label{peanoprop}
If $X$ and each $B_j\in\scrb$ is a compact Haudsorff space (respectively, a compact metric space, an $n$-dimensional compact metric space, a Peano continuum, an $n$-dimensional Peano continuum), then so is $\san(X,A,\scrb)$. 
\end{proposition}

\begin{proof}
Define $Z_m$ as above. Since each $X$ and $B_j$ is compact Hausdorff, so is each $Z_m$. Thus the inverse limit $\varprojlim_{m}Z_m$ is compact Hausdorff. By Corollary \ref{invlimitcor}, $\san(X,A,\scrb)\cong \varprojlim_{m}Z_m$. Thus $\san(X,A,\scrb)$ is compact Hausdorff. If, in addition, $X$ and each $B_j$ are metrizable, then each $Z_m$ is metrizable. Since limits of inverse sequences are closed under metrizability, it follows that $\varprojlim_{m}Z_m$ is a compact metric space. 

Suppose $X$ and each $B_j$ is a Peano continuum. By the previous paragraph $\san(X,A,\scrb)$ is a compact metric space. By Proposition \ref{propertiesprop}, $\san(X,A,\scrb)$ is path-connected and locally path-connected. Thus $\san(X,A,\scrb)$ is a Peano continuum.

Lastly, suppose $X$ and each $B_j$ is a compact metric space of dimension $n$ (recall that under these hypotheses, the small inductive, large inductive, and covering dimensions agree). Since $X,B_1,B_2,B_3,\dots$ is a cover of $\san(X,A,\scrb)$ by $n$-dimensional spaces, the countable sum theorem  \cite[Theorem 4.1.9]{EngelkingDimThry} applies and we may conclude that $\san(X,A,\scrb)$ is a $n$-dimensional compact metric space.
\end{proof}

With several topological issues involving shrinking point-attachment spaces settled, we study the wild set of $\san(X,A,\scrb)$. Recall that we may use $A$ to denote the image of the sequence of attachment points in $X$.

\begin{lemma}\label{sanlemma}
Suppose $X$ is a Peano continuum and each $B_j$ is a non-simply connected Peano continuum. Then
\[\wildn(X)\cup A'\cup \bigcup_{j\in\bbn}\wildn(B_j)\subseteq\wildn(\san(X,A,\scrb))\subseteq\wildn(X)\cup \ov{A}\cup\bigcup_{j\in\bbn}\wildn(B_j)\]
where $A'$ denotes the set of limit points of $A$ in $X$.
\end{lemma}

\begin{proof}
Since finite and shirking wedges of non-simply connected spaces are not simply connected, we may assume that $A$ is injective and write $A=\{a_1,a_2,a_3,\dots\}$. Note that $A'$ may not contain $A$ as a subset if $A$ has isolated points. Since $X$ and each $B_j$ is a retract of $\san(X,A,\scrb)$, Corollary \ref{retractcor} gives $\wildn(X)\cup \bigcup_{j\in\bbn}\wildn(B_j)\subseteq \wildn(\san(X,A,\scrb))$. If $x\in A'$, find $j_1<j_2<j_3<\cdots$ such that $\{a_{j_i}\}_{i\in\bbn}$ converges to $x$. For each $i\in\bbn$, find an essential loop $\beta_i:\ui\to B_{j_i}$ based at $b_{j_i}$. Find a sequence of paths $\alpha_i:\ui\to X$ from $x$ to $a_{j_i}$ such that $\{\alpha_i\}_{i\in\bbn}$ converges to $x$. Define $f:(\bbe_n,b_0)\to (\san(X,A,\scrb),x)$ so that $f\circ \ell_i$ is the path-conjugate of $\beta_i$ by $\alpha_i$. Since $B_{j_i}$ is a retract of $X$ for each $i$, the map $f$ is full essential. Thus $x\in \wildn(\san(X,A,\scrb))$. This completes the proof of the first inclusion.

For the second inclusion, note that for each $j\in\bbn$, we can write $\san(X,A,\scrb)=Y_j\vee B_j$ with wedgepoint $b_j$. Corollary \ref{wedgecor} gives $\wildn(\san(X,A,\scrb))\subseteq\wildn(Y_j)\cup \wildn(B_j)\cup \{b_j\}$. Thus, $\wildn(\san(X,A,\scrb))\cap B_j\subseteq\wildn(B_j)\cup\{b_j\}\subseteq \wildn(B_j)\cup A$ for all $j\in\bbn$. To finish the proof, it suffices to show that $\wildn(\san(X,A,\scrb))\cap X\subseteq\wildn(X)\cup\ov{A}$. Since $X$ and $\san(X,A,\scrb)$ are Peano continua, Lemma \ref{closedprop} implies that both $\wildn(\san(X,A,\scrb))\cap X$ and $\wildn(X)\cup\ov{A}$ are closed in $X$. In particular, $U=X\backslash (\wildn(X)\cup\ov{A})$ is open in $X$. Since $U$ does not contain any subsequential limit of attachment points, $U$ vacuously satisfies Conditions (2) and (3) of Definition \ref{shrinkingattachmentdef} and thus $U$ is open in $\san(X,A,\scrb)$. If $x\in (\wildn(\san(X,A,\scrb))\cap X)\backslash (\wildn(X)\cup\ov{A})$, then $x\in U$ and we can find a fully essential map $f:(\bbe_n,b_0)\to (\san(X,A,\scrb),x)$ with restriction $f_k=f\circ\ell_k$ to the $k$-th sphere. Since $U$ is open in $\san(X,A,\scrb)$, we may restrict $f$ to a cofinal sequence of spheres and, therefore, assume that $f(\bbe_n)\subseteq U$. Since each $f_k$ has image in $X$ and is essential in $\san(X,A,\scrb)$, each $f_k$ must be essential in $X$. Therefore, $x\in \wildn(X)$; a contradiction. This completes the proof of the second inclusion.
\end{proof}

\begin{corollary}\label{sancor}
Suppose $X$ is a Peano continuum and each $B_j$ is a non-simply connected Peano continuum. If $b_j\in\wildn(B_j)$ for each $j\in\bbn$, then $\wildn(\san(X,A,\scrb))=\wildn(X)\cup \ov{A}\cup \bigcup_{j\in\bbn}(\wildn(B_j))$. Moreover, if $A$ is dense in $X$, then $\wildn(\san(X,A,\scrb))=X\cup \bigcup_{j\in\bbn}\wildn(B_j)$.
\end{corollary}

\begin{proof}
If $b_j\in \wildn(B_j)$ for each $j\in\bbn$, then $A\subseteq \bigcup_{j\in\bbn}\wildn(B_j)$ Thus $\ov{A}=A\cup A'\subseteq \san(X,A,\scrb)$ by the first inequality of Lemma \ref{sanlemma}. Applying the second inequality from Lemma \ref{sanlemma} completes the proof.
\end{proof}

\begin{remark}
The purpose of the $\pi_n$-wildification construction is to make each point of $A$ a $\pi_n$-wild point if it is not one already. We choose to use the space $B=\bbe_n$ instead of $S^n$ in our definition of $\pi_n$-wildification because the image of the sequence $A$ may have isolated points. In particular, if $a$ is an isolated point of $\im(A)$, $A^{-1}(a)$ is finite, and $a\in X\backslash\wildn(X)$, then $a$ will not be a $\pi_n$-wild point of $\san(X,A,S^n)$. However, in the case that $A$ is dense, Lemma \ref{sanlemma} implies that $\wildn(\san(X,A,S^n))=\wildn(\san(X,A,\bbe_n))=X$. In fact, the following can be proved with modest effort: If $X$ is a Peano continuum and $A\subseteq X$ is dense, then there is a homotopy equivalence $f:\san(X,A,S^n)\to \san(X,A,\bbe_n)$ that is the identity on $X$. We do not require this result and the proof is a divergence from our focus on $\pi_n$-wild sets so we do not give it here. 
\end{remark}

\begin{lemma}\label{peanocontlemma2}
For every finite-dimensional Peano continuum $Y$ and closed subspace $X\subseteq Y$, there exists a Peano continuum $Z$ such that
\begin{enumerate}
\item $Y$ is a retract of $Z$ where $Z\backslash Y$ is a disjoint union of countably many open $n$-cells,
\item $\wildn(Z)=X\cup \wildn(Y)$,
\item $\dim(Z)= \max\{n,\dim(Y)\}$.
\end{enumerate}
\end{lemma}

\begin{proof}
Let $A$ be a countable dense subset of $X$ and $Z=\san(Y,A,\bbe_n)$ be the $\pi_n$-wildification as described in Section \ref{section5makingspaceswild}. (1) is clear from the construction of $\san(Y,A,\bbe_n)$. (2) follows from Lemma \ref{sanlemma}. For (3), we note that $\max\{n,\dim(Y)\}\leq \dim(Z)$ since $Z$ is the union of $Y$ and open $n$-cells. Another application of the Sum Theorem \cite[1.5.3]{EngelkingDimThry} gives $\dim(Z)\leq \max\{n,\dim(Y)\}$.
\end{proof}

\begin{proof}[Proof of Theorem \ref{thm1}]
Using Lemma \ref{peanocontlemma}, find a Peano continuum $Y$ such that $X\subseteq Y$, $Y\backslash X$ is a countable disjoint union of open $1$-cells and $\dim(Y)= \max\{1,\dim(X)\}$. Applying the construction in the proof of \ref{peanocontlemma2}, we obtain a Peano continuum $Z=\san(Y,A,\bbe_n)$ where $Z\backslash Y$ is a disjoint union of countably many open $n$-cells, $Y$ is a retract of $Z$, $\wildn(Z)=X$, and $\dim(Z)= \max\{n,\dim(Y)\}$. Our use of $\pi_n$-wildification ensures that Conclusions (1) and (2) hold. Since $\dim(Z)= \max\{n,\dim(Y)\}=\max\{n,1,\dim(X)\}=\max\{n,\dim(X)\}$, (3) holds.
\end{proof}

Whenever $Z$ is a Peano continuum, $\wildn(Z)$ is a compact metric space by Corollary \ref{closedcor}. Therefore, we have the following.

\begin{corollary}
A space $X$ is an ($n$-dimensional) compact metric space if and only if it is the $\pi_n$-wild set of some ($n$-dimensional) Peano continuum.
\end{corollary}

While the $\pi_n$-wild set of a compact metric space must be separable, recall from Example \ref{sinecurveexample} that it need not be compact. The authors do not know if every separable metric space is the $\pi_n$-wild set of some compact metric space. 

\begin{problem}
Characterize the class of spaces consisting of $\pi_n$-wild sets of compact metrizable spaces.
\end{problem}

\section{Rigidity of wild sets}\label{section6rigidity}

Corollary \ref{tpdcor} identifies a situation where the \textit{homeomorphism type} of $\wildn(X)$ is an invariant of the homotopy type of $X$. The same type of rigidity is also known to occur in other situations.

\begin{definition}\label{rigid}
We say that a space $X$ is \textit{$\pi_n$-rigid at} $x\in X$ if there exists a fully essential map $f:(\bbe_n,b_0)\to (X,x)$ with the property that for any map $F:\bbe_n\times \ui\to X$ extending $f$ by $F(a,0)=f(a)$, we have $F(b_0,1)=x$. Let \[\rgn(X)=\{x\in X\mid X\text{ is }\pi_n\text{-rigid at }x\}.\]
We say a space $X$ is \textit{completely $\pi_n$-rigid} if $\rgn(X)=\wildn(X)$.
\end{definition}

Intuitively, we have $x\in \rgn(X)$ if there exists a fully essential map $f:(\bbe_n,b_0)\to (X,x)$ that cannot be freely homotoped in a fashion that moves the basepoint away from $x$. Certainly, $\rgn(X)\subseteq \wildn(X)$.


\begin{example}
If $\wildn(X)$ is non-empty and totally path-disconnected, then $\rgn(X)=\wildn(X)$. Indeed, in such a space, any map $g:\bbe_n\times \ui\to X$ for which $g(x,0):\bbe_n\to X$ is fully essential must map $\{b_0\}\times \ui$ to $g(b_0,0)$. In particular, $g(b_0,1)=g(b_0,0)$ showing $x\in \rgn(X)$. 
\end{example}

\begin{remark}
In \cite[Definition 9.2]{BFPantsspace}, a space $X$ is said to have the \textit{discrete monodromy (DM) property} if for every path $\beta:\ui\to X$ that is not an inessential loop, there exists neighborhoods $U$ of $\beta(0)$ and $V$ of $\beta(1)$ such that if $\gamma\in \Omega(U,\beta(0))$ and $\delta\in \Omega(V,\beta(1))$ satisfy path-homotopy relation $\gamma\simeq \beta\cdot\delta\cdot \beta^{-}$, then $\gamma$ and $\delta$ must be inessential. Certainly, the DM-property implies the completely $\pi_1$-rigid property. However, the DM Property also implies the homotopically Hausdorff property \cite[Corollary 9.12]{BFPantsspace} because of the non-trivial case where $\beta$ is an \textit{essential} loop. On the other hand, the completely $\pi_1$-rigid property does not imply the homotopically Hausdorff property. For example, the Griffths double cone (Example \ref{wedgeexample}) is not homotopically Hausdorff but has a single $\pi_1$-wild point and is therefore is completely $\pi_1$-rigid according to the previous example.
\end{remark}

\begin{proposition}\label{dmp1}
Suppose $H:X\times \ui\to Y$ is a map such that $f:X\to Y$, $f(x)=H(x,0)$ is $\pi_n$-injective. If $x_0\in\wildn(X)$ and $y_0=f(x_0)\in \rgn(Y)$, then $H(x_0,0)=H(x_0,1)$.
\end{proposition}

\begin{proof}
Since $x_0\in\wildn(X)$, there exists a fully essential map $\alpha:(\bbe_n,b_0)\to (X,x_0)$. Consider $G=H\circ (\alpha\times id_{\ui}):\bbe_n\times\ui\to Y$. Let $g:\bbe_n\to Y$ be the map $g(a)=G(a,0)$. Since $g=f\circ \alpha$ where $f$ is $\pi_n$-injective, $g$ is fully essential. Since $y_0\in \rgn(Y)$, we conclude that $G(b_0,1)=y_0$. Thus $H(x_0,0)=y_0=G(b_0,1)=H(x_0,1)$.
\end{proof}

\begin{corollary}
If $k:X\to X$ is homotopic to the identity map, then $k(x)=x$ for all $x\in\rgn(X)$.
\end{corollary}

\begin{proof}
Suppose $H:X\times \ui\to X$ is a map such that $f(a)=H(a,0)=a$ is the identity map and $H(a,1)=k(a)$. Let $x\in \rgn(X)$. Since $f(x)=x$, we have $x=H(x,0)=H(x,1)=k(x)$ by Proposition \ref{dmp1}.
\end{proof}

\begin{corollary}\label{hicor}
If $f:X\to Y$ and $g:Y\to X$ are homotopy inverses, then $g\circ f(x)=x$ for all $x\in\rgn(X)$ and $f\circ g(y)=y$ for all $y\in\rgn(Y)$.
\end{corollary}

Note that Corollary \ref{hicor} does not imply that $f$ restricts to a homeomorphism on $\pi_n$-rigid wild sets because it need not be the case that $f$ maps $\rgn(X)$ into $\rgn(Y)$, e.g. if $f$ is the embedding $\bbe_n\to \bbe_n\times\ui$, $x\mapsto (x,0)$, then $\rgn(\bbe_n)=\{b_0\}$ and $\rgn(\bbe_n\times\ui)=\emptyset$. However, when $X$ is completely $\pi_n$-rigid, we have the following consequence.

\begin{theorem}\label{rigidthm1}
If $X$ and $Y$ are homotopy equivalent completely $\pi_n$-rigid spaces, then $\wildn(X)$ and $\wildn(Y)$ are homeomorphic.
\end{theorem}

\begin{proof}
If $f:X\to Y$ and $g:Y\to X$ are homotopy inverses, then $f(\wildn(X))\subseteq \wildn(Y)$ and $g(\wildn(Y))\subseteq \wildn(X)$. By hypothesis, we have $\rgn(X)=\wildn(X)$ and $\rgn(Y)=\wildn(Y)$. Thus $f(\rgn(X))\subseteq \rgn(Y)$ and $g(\rgn(Y))\subseteq \rgn(X)$ and Corollary \ref{hicor} implies that $f|_{\rgn(X)}:\rgn(X)\to \rgn(Y)$ and $g|_{\rgn(Y)}:\rgn(Y)\to \rgn(X)$ are inverse homeomorphisms.
\end{proof}

As mentioned previously, the following theorem is implicit in the work of Eda \cite{Edaonedim} and Conner-Kent \cite{ConnerKent} and follows explicitly from the results in \cite[Prop. 9.13]{BFPantsspace} and the fact that the DM-Property implies the completely $\pi_1$-rigid property.

\begin{theorem}\label{onerigidtheorem}
If $X$ is a one-dimensional metric space or planar set, then $X$ is completely $\pi_1$-rigid.
\end{theorem}


All one-dimensional metric spaces have the property that their fundamental groups canonically inject into their first shape homotopy group and this fact plays a role in the proof of Theorem \ref{onerigidtheorem}. Thus, if we are searching for a higher dimensional analogue of Theorem \ref{onerigidtheorem}, it is natural to consider spaces for which their $n$-th homotopy group canonically embeds into the $n$-th homotopy shape group. We refer to \cite{MS82} for the foundations of Shape Theory and to \cite{AcetiBrazas} for an explicit description of the canonical homeomorphism $\Psi_n:\pi_n(X)\to \check{\pi}_n(X)$ defined in terms of the Cech expansion of $X$. A based space $X$ is \textit{$\pi_n$-shape injective} if $\Psi_n$ is injective. Since our focus is on compact metric spaces, we only require the following simpler description of $\Psi_n$: Suppose $(X,x_0)=\varprojlim_{i\in\bbn}((K_i,k_i),q_{i+1,i})$ is an inverse limit of compact polyhedra $K_i$ and based continuous functions $q_{i+1,i}:K_{i+1}\to K_i$. Let $q_i:X\to K_i$ be the projection maps. Then because $X$ is compact Hausdorff, the system $((K_i,k_i),q_{i+1,i})$ serves as an $HPol_{\ast}$-expansion of $(X,x_0)$ \cite[Ch I, \textsection 5.4, Theorem 13]{MS82}. Thus $X$ is $\pi_n$-shape injective if and only if the canonical homomorphism $\Psi_n:\pi_n(X,x_0)\to \varprojlim_{i\in\bbn}\pi_n(K_i,k_i)$, $\Psi_n([f])=([q_i\circ f])_{i\in\bbn}$ is injective.

The next lemma is proved using standard methods in homotopy theory so we give a brief sketch of the argument.

\begin{lemma}\label{freelyhomlemma}
Suppose $n\geq 1$ and $P$ is a $n$-dimensional polyhedron with an $(n-1)$-connected universal cover $\wt{P}$. If $f:S^n\to P$ and $g:S^n\to P$ are freely homotopic maps with disjoint images, then $f$ and $g$ are inessential.
\end{lemma}

\begin{proof}
Find simplicial complex $K$ with $|K|\cong P$ and identify these spaces. By taking a sufficiently fine subdivision of $K$, we may assume that there are disjoint subcomplexes $A$ and $B$ of $K$ such that $\im(f)\subseteq |A|$ and $\im(g)\subseteq |B|$. The case $n\geq 1$ follows from the fact that in a free product $G_1\ast G_2$, no non-trivial element of $G_1$ is conjugate to an element of $G_2$. Supposing $n\geq 2$, we may replace $|K|$ with it's universal cover. Thus we may assume that $|K|$ is $(n-1)$-connected and by the Hurewicz Theorem, we have $\pi_n(|K|)\cong H_n(|K|)$. Find open sets $U,W$ with $|A|\subseteq \ov{W}\subseteq U\subseteq |K|\backslash |B|$ such that if $V=|K|\backslash \ov{W}$, then $U$ deformation retracts onto $|A|$ and $U\cap V$ deformation retracts on to a subpolyhedron of dimension at most $n-1$. Since $H_{n}(U\cap V)=0$, the subtraction map $d:H_n(U)\oplus H_n(V)\to H_n(|K|)$ from the Meyer-Vietoris sequence is injective. Consider the $n$-dimensional homology classes $\alpha\in H_n(U)$ and $\beta\in H_n(V)$ corresponding to $f$ and $g$ respectively. Since $f$ and $g$ are freely homotopic, $i_{\ast}(\alpha)=j_{\ast}(\beta)$ in $H_n(|K|)$ where $i:U\to |K|$ and $j:V\to |K|$ are the inclusions. Thus $d(\alpha,\beta)=i_{\ast}(\alpha)-j_{\ast}(\beta)=0$ and the injectivity of $d$ gives $\alpha=\beta=0$. We conclude that $f$ and $g$ give trivial homology classes in $H_n(|K|)$ and, consequently, trivial homotopy classes in $\pi_n(|K|)$.
\end{proof}

\begin{proof}[Proof of Theorem \ref{thm2}]
Suppose $X=\varprojlim_{i\in\bbn}(K_i,q_{i+1,i})$ is a $\pi_n$-shape injective inverse limit of a sequence of based $n$-dimensional compact polyhedra $K_i$ with $(n-1)$-connected universal covers $\wt{K_n}$. Let $q_i:X\to K_i$ be the projection maps. Since $X$ is compact Hausdorff, the system $(K_i,p_{i+1,i})$ serves as an $HPol_{\ast}$-expansion and, as mentioned above, the canonical homomorphism $\Psi_n:\pi_n(X)\to \varprojlim_{i\in\bbn}\pi_n(K_i)$, $\Psi_n([f])=([q_i\circ f])_{i\in\bbn}$ is injective.

To obtain a contradiction, suppose that $X$ is not completely $\pi_n$-rigid. Then there exists a map $H:\bbe_n\times\ui\to X$ such that $a_0=H(b_0,0)\neq H(b_0,1)=a_1$ and such that if $f_t:\bbe_n\to X$ is defined by $f_t(a)=H(a,t)$, then the maps $f_0,f_1:\bbe_n\to X$ are fully essential. Let $f_{t,j}:S^n\to X$ denote the $j$-th restriction of $f_t$. 

Find an $i_1\in\bbn$ such that $q_{i_1}(a_0)\neq q_{i_1}(a_1)$. Let $U_0$ and $U_1$ be disjoint neighborhoods of $q_{i_1}(a_0)$ and $q_{i_1}(a_1)$ in $K_{i_1}$. Find $J\in\bbn$ such that $\im(q_{i_1}\circ f_{0,J})\subseteq U_0$ and $\im(q_{i_1}\circ f_{1,J})\subseteq U_1$. Since $\Psi_n$ is injective, we may find $i_2\geq i_1$ such that $q_{i_2}\circ f_{0,J}:S^n\to K_{i_2}$ and $q_{i_2}\circ f_{1,J}:S^n\to K_{i_2}$ are essential. Let $q_{i_2,i_1}=q_{i_1+1,i_1}\circ q_{i_1+2,i_1+1}\circ \cdots \circ q_{i_2,i_2-1}$. Since $q_{i_2,i_1}\circ q_{i_2}=q_{i_1}$, the images of the maps $q_{i_2}\circ f_{0,J}$ and $q_{i_2}\circ f_{1,J}$ lie in the disjoint sets $q_{i_2,i_1}^{-1}(U_0)$ and $q_{i_2,i_1}^{-1}(U_1)$ respectively. Note that $G:S^n\times \ui\to K_{i_2}$, $G(x,t)=q_{i_2}\circ f_{t,J}(x)$ gives a free homotopy between $q_{i_2}\circ f_{0,J}$ and $q_{i_2}\circ f_{1,J}$. But the maps $q_{i_2}\circ f_{0,J}$ and $q_{i_2}\circ f_{1,J}$ are essential and have disjoint images, contradicting Lemma \ref{freelyhomlemma}.
\end{proof}

\begin{example}
For $i\in\{1,2\}$, let $X_i$ be a one-dimensional Peano continuum and let $A_i\subseteq X_i$ be a countable dense set. The $\pi_n$-wildification $Y_i=\san(X_i,A_i,\bbe_n)$ is a n-dimensional Peano continuum (Proposition \ref{peanoprop}) with $\pi_n$-wild set $\wildn(Y_i)=X_i$ (Lemma \ref{sanlemma}). In forthcoming work, the first author has used generalized covering space theory to show that spaces of the form $Y_i$ are $\pi_n$-shape injective. Moreover, since $X_i$ is an inverse limit of finite graphs, the space $Y_i$ is an inverse limits of finite graphs with finitely many $n$-spheres attached. Such approximating spaces satisfy the hypotheses of Theorem \ref{thm2}. Thus $Y_1,Y_2$ are completely $\pi_n$-rigid. This allows us to produce continuum-many distinct homotopy types of $2$-dimensional Peano continua since if the one-dimensional spaces $X_1$ and $X_2$ are not homeomorphic, then the resulting $\pi_n$-wildification spaces $Y_1$ and $Y_2$ are not homotopy equivalent.
\end{example}

\begin{example}\label{connectednessexample}
Here, we illustrate the importance of the higher connectedness hypothesis in Theorem \ref{thm2}. The space $\bbe_{2}\times\ui$ is a simply connected, $3$-dimensional Peano continuum. However, $\bbe_{2}\times\ui$ does not meet the connectedness hypothesis of Theorem \ref{thm2} (approximating polyhedra have non-trivial $\pi_2$). Moreover, $\bbe_{2}\times\ui$ is not completely $\pi_3$-rigid. Indeed, let $\eta:S^3\to S^2$ be the Hopf fibration. We may define a map $f:\bbe_3\to\bbe_2$ so that the $j$-th restriction is $f_j=\eta$ for all $j$. Then the map $f\times id_{\ui}:\bbe_3\times \ui\to \bbe_{2}\times\ui$ shows that $\wild_3(\bbe_{2}\times\ui)=\{b_0\}\times \ui$ and $\rg_3(\bbe_{2}\times\ui)=\emptyset$. 
\end{example}





\end{document}